\documentclass[reqno]{amsart}
\usepackage{hyperref}
\usepackage{amssymb}

\begin{document}
\title[\hfilneg 2019\hfil Renormalized solution to the Riccati type equation]
{An application of global gradient estimates in Lorentz-Morrey spaces: The existence of
stationary solution to degenerate diffusive Hamilton-Jacobi equations}

\author[M.-P. Tran, T.-N. Nguyen \hfil EJDE-2018\hfilneg]
{Minh-Phuong Tran, Thanh-Nhan Nguyen}

\address{Minh-Phuong Tran (corresponding author) \newline
Applied Analysis Research Group, Faculty of Mathematics and Statistics, Ton Duc Thang University, Ho Chi Minh city, Vietnam}
\email{tranminhphuong@tdtu.edu.vn}

\address{Thanh-Nhan Nguyen \newline
Department of Mathematics, Ho Chi Minh City University of Education, Ho Chi Minh city, Vietnam}
\email{nguyenthnhan@hcmup.edu.vn}



\begin{abstract}

In historical mathematics and physics, the Kardar-Parisi-Zhang equation or a quasilinear stationary version of a time-dependent viscous Hamilton-Jacobi equation in growing interface and universality classes, is also known by the different name as the quasilinear Riccati type equation. The existence of solutions to this type of equation under some assumptions and requirements, still remains an interesting open problem at the moment. In our previous studies \cite{MP2018, MPT2019}, we obtained the global bounds and gradient estimates for quasilinear elliptic equations with measure data. There have been many applications are discussed related to these works, and main goal of this paper is to obtain the existence of a renormalized solution to the quasilinear stationary solution to the degenerate diffusive Hamilton-Jacobi equation with the finite measure data in Lorentz-Morrey spaces. 
\end{abstract}

\maketitle
\newtheorem{theorem}{Theorem}[section]
\newtheorem{corollary}[theorem]{Corollary}
\newtheorem{definition}[theorem]{Definition}
\newtheorem{example}[theorem]{Example}
\newtheorem{lemma}[theorem]{Lemma}
\newtheorem{proposition}[theorem]{Proposition}
\newtheorem{remark}[theorem]{Remark}


\section{Introduction}
\label{sec:intro}

This paper is devoted to consider the existence of renormalized solution of the following stationary degenerate diffusive Hamilton-Jacobi equation, with respect to a given measure data $\mu$, that has the form:
\begin{equation}
\label{eq:Riccati}
\begin{cases}
-\mbox{div}(A(x,\nabla u)) &= \ |\nabla u|^q + \mu \quad \text{in} \ \ \Omega, \\
\hspace{1.8cm} u &= \ 0 \quad \text{on} \ \ \partial \Omega,
\end{cases}
\end{equation}
in Lorentz-Morrey spaces $L^{s,t;\kappa}(\Omega)$ (the optimal range of $s,t$ and $\kappa$ will be clarified in our proof later). It is noticeable that our domain $\Omega \subset \mathbb{R}^n$ ($n \ge 2$) is a bounded domain whose complement satisfies a $p$-capacity uniform thickness condition. Specifically and precisely, in the present work, we consider for extended case, in which $p \in \left(\frac{3n-2}{2n-1},n\right)$. Moreover, in our problem, the nonlinearity  $A: \ \Omega \times \mathbb{R}^n \to \mathbb{R}^n$ is a Carath{\'e}dory vector valued function which satisfies growth and monotonicity conditions, i.e., there exist positive constants $c_1, c_2$ such that for some $p >1$ there holds
\begin{align*}
\left| A(x,\xi) \right| &\le c_1 |\xi|^{p-1},\\
\langle A(x,\xi_1)-A(x,\xi_2), \xi_1 - \xi_2 \rangle &\ge c_2 \left( |\xi_1|^2 + |\xi_2|^2 \right)^{\frac{p-2}{2}}|\xi_1 - \xi_2|^2,
\end{align*}
for every $\xi,\, \xi_1,\, \xi_2 \in \mathbb{R}^n \setminus \{0\}$ and $x \in \Omega$ almost everywhere. 

This type of equation often appears in physical theory of surface growth, also known as the Kardar-Parisi-Zhang (KPZ) equation, where the study of this equation is still a challenge for mathematicians. It can be viewed as a quasilinear stationary version of a time-dependent viscous Hamilton-Jacobi equation, and it would be applied much in growing interface and universality classes (see \cite{KPZ, KS}). Specifically, for the case of $A(x,\xi) = |\xi|^{p-2}\xi$, the considered equation~\eqref{eq:Riccati} is a type of standard $p$-Laplace equation
$$
-\Delta_p u  = \ |\nabla u|^q + \mu,
$$
and this equation has been studied extensively by several authors  with  their fine papers \cite{BGV,HM1999,Martio}, in both historical view of mathematics and physics. Since then, for the general nonlinearity $A$, much attention has been devoted to the existence of solution also some comparison estimates, regularity theories of the problem. There have been several studies to the existence of solution to~\eqref{eq:Riccati} under different assumptions, and later extended to several spaces. More precisely, it was mentioned in \cite[page 13-14]{BGV} about the sharp existence for the $p$-Laplacian problem in supercritical case. And later, in many works of O. Martio~\cite{Martio2}, Mengesha {\it et al.}~\cite{MePh}, N.C.~Phuc {\it et al.} (see \cite{ MePh, 55Ph1,  55Ph2}) and M.-P.~Tran {\it et al.} (see~\cite{PN2019}), it is also related to the existence of renormalized solution to \eqref{eq:Riccati} under different hypotheses of domain $\Omega$, the nonlinearity operator $A$ and the functional spaces. Motivated by these works, we are interested in the solvability to equation~\eqref{eq:Riccati} in Lorentz-Morrey spaces for the supercritical case $q \in \left( \frac{n(p-1)}{n-1}, p \right)$ under the $p$-capacity uniform thickness condition of the domain $\Omega$. 

There are several tools developed for linear and/or nonlinear potential and Calder\'on-Zygmund theories in recent years (see \cite{BW1, 11DMOP, 55DuzaMing, Duzamin2, 55MePh2, Mi2, 55QH2, 55Ph0, 55Ph2}). It is worth pointing out that in our study, the key ingredients were based on some local comparison estimates of renormalized solution to the following quasilinear elliptic equation:
\begin{equation}
\label{eq:elliptictype}
\begin{cases}
-\mbox{div}(A(x,\nabla u))&= \ \mu \quad \text{in} \ \ \Omega, \\
\hspace{2cm} u &=\ 0 \quad \text{on} \ \ \partial \Omega.
\end{cases}
\end{equation} 

Earlier, there were a series of works by G. Mingione {\it et al.} (in~\cite{55DuzaMing},~\cite{Duzamin2},~\cite{KMi1}, \cite{KMi2},\cite{Mi2}~\cite{Mi3}), N.~C.~Phuc {\it et al.} (in \cite{AdP1, 55Ph0,  55Ph1, 55Ph2}), Q.~H.~Nguyen {\it et al.} (in \cite{55QH2, 55QH3, 55QH4, 55QH5} and references therein), M.-P. Tran {\it et al.} (in~\cite{MP2018, MPT2019}), in which authors gave a local and global gradient estimates in Lorentz or Morrey-Lorentz spaces under various assumptions on $\Omega$. 

In the advantage of using the hypothesis of $p$-capacity uniform thickness condition in~\cite{55Ph1}, the gradient estimate of renormalized solution to~\eqref{eq:elliptictype} were known for the regular case of $ p \in \left( 2 - \frac{1}{n}, n\right)$. And in our previous work~\cite{MPT2019}, we established the Lorentz-Morrey global bound for quasilinear elliptic equation~\eqref{eq:elliptictype} in singular case of $p \in \left(\frac{3n-2}{2n-1}, 2-\frac{1}{n}\right]$. The Morrey global bound for equation~\eqref{eq:elliptictype} in the singular case is also studied in~\cite{55QH4} under hypotheses of Reifenberg domain $\Omega$ and smallness BMO of operator $A$. In the present paper, as an application of global gradient estimates studied in~\cite{MPT2019}, we will make a discussion of the solvability of equation~\eqref{eq:Riccati} in Lorentz-Morrey spaces for singular cases with only the hypothesis of $p$-capacity uniform thickness condition. However, we connect the estimates in~\cite{55Ph1} and~\cite{MPT2019} to obtain a complete existence result for both regular and singular cases, that is the reason why we generalize our result for $p \in \left(\frac{3n-2}{2n-1},n\right)$. 

We first recall the Lorentz-Morrey global bounds of renormalized solution to equation \eqref{eq:elliptictype}, that was proved in~\cite{55Ph1} and~\cite{MPT2019}. The following theorem is obtained by combining the gradient estimate results for the regular case in~\cite[Theorem 1.1]{55Ph1} and the singular case in~\cite[Theorem 1.1]{MPT2019}. We notice that the quasi-norm $\| \cdot \|_{L^{s,t; \,\kappa}(\Omega)}$ in Lorentz-Morrey space $L^{s,t; \,\kappa}(\Omega)$ will be presented in the next section.

\begin{theorem}
\label{coro:P}
 Let $n \ge 2$,  $ p \in \left(\frac{3n-2}{2n-1}, n\right)$ and $\Omega \subset \mathbb{R}^n$ be a bounded domain whose complement satisfies a $p$-capacity uniform thickness condition. Assume that $\mu \in L^{\frac{s(\theta-1)}{\theta(p-1)}, \frac{t(\theta-1)}{\theta(p-1)};\frac{s(\theta-1)}{p-1}}(\Omega)$ for some $s\in (0,p]$, $t \in (0,\infty]$ and $\theta \in [p, n]$. Then for any renormalized solution $u$ to equation~\eqref{eq:elliptictype}, there exists a positive constant $C$ such that
\begin{align}
\label{eq:apply}
&\|\nabla u\|_{L^{s,t;\frac{s(\theta-1)}{p-1}}(\Omega)} \leq C \|\mu\|^{\frac{1}{p-1}}_{L^{\frac{s(\theta-1)}{\theta(p-1)}, \frac{t(\theta-1)}{\theta(p-1)};\frac{s(\theta-1)}{p-1}}(\Omega)}.
\end{align}
\end{theorem}
In this paper, we prove the existence result of a renormalized solution to equation~\eqref{eq:Riccati} in Lorentz-Morrey space for both singular and regular cases $p\in \left( \frac{3n-2}{2n-1}, n\right)$ in the super-critical case $q > \frac{n(p-1)}{n-1}$. Our proof is based on applying Theorem~\ref{coro:P} and the Schauder Fixed Point Theorem in~\cite{Trudinger}. The main idea of this proof comes from the proof of the existence result studied in~\cite{MePh}. More precisely, we consider a closed and convex set $S$ as the form
$$ S = \left\{v \in W_0^{1,1}(\Omega): \ |||\nabla v|^q||_{L^{s,t;\frac{sq(\theta-1)}{p-1}}(\Omega)} \le \varepsilon\right\},$$
where the positive constant $\varepsilon$ is chosen later. We note that the convexity of $S$ will be obtained for $qs>1$. For every $v\in S$, we define by $T(v) = u$ the unique renormalized solution to the following equation 
\begin{equation*}
\begin{cases}
-\mbox{div}(A(x,\nabla u)) &=\ |\nabla v|^q+ \mu \quad \text{in} \ \ \Omega, \\
\hspace{2cm} u &= \ 0 \quad \text{on} \ \ \partial \Omega.
\end{cases}
\end{equation*}
We refer to~\cite{11DMOP} for the uniqueness of renormalized solution to above equation.
By Theorem~\ref{coro:P}, we can prove that the mapping $T: \ S \to S$ is well-defined, continuous and $T(S)$ is precompact under the strong topology of $W_0^{1,1}(\Omega)$. The existence result can be obtained by the Schauder Fixed Point Theorem. Let us state our main result in the following theorem.
\begin{theorem}
\label{theo:main}
Let  $n \ge 2$, $ p\in \left( \frac{3n-2}{2n-1}, n\right)$ and $\Omega \subset \mathbb{R}^n$ be a bounded domain whose complement satisfies a $p$-capacity uniform thickness condition. Assume that
\begin{equation}
\label{eq:cond_q}
\max\left\{\frac{n(p-1)}{n-1}, p-1 + \frac{1}{n}\right\}  < q < p.
\end{equation}
 For any $q \le t \le \infty$  and 
 \begin{equation}
 \label{eq:cond_s}
 \max\left\{1,\frac{1}{q}\right\} <s \le \min\left\{\frac{p}{q},\frac{n}{\theta}\right\},
 \end{equation}
with $\theta = \frac{q}{q-p+1}$. There exists $\delta_0>0$ such that if $\|\mu\|_{L^{s,t;\,\theta s}(\Omega)}\leq \delta_0$ then the equation \eqref{eq:Riccati} admits a renormalized solution $u$ satisfying 
\begin{align}
\label{eq:est}
\|\nabla u\|^q_{L^{qs,qt;\,\theta s}(\Omega)} \le \theta\delta_0 - \|\mu\|_{L^{s,t;\,\theta s}(\Omega)}.
\end{align}
\end{theorem}

The rest of the paper is organized as follows. In the next section, we recall the definitions of Lorentz and Lorentz-Morrey spaces.  Moreover, we introduce a norm which is equivalent to the quasi-norm in Lorentz-Morrey spaces. The proof of Theorem~\ref{theo:main} is given in the last section.

\section{Lorentz-Morrey spaces}
\label{sec:pre}
 In this section, we give some backgrounds about the definitions of Lorentz and Lorentz-Morrey spaces equipped to an usual quasi-norm in general. The nice feature is that this quasi-norm is equivalent to a norm in these functional spaces (see~\cite{55Gra}). In this paper, we give a simple proof for the equivalence between two norms which is useful for our proof in the next section. We assume that $\Omega$ is an open bounded subset of $\mathbb{R}^n$ with $n \ge 2$. For convenience of the reader, we first recall the definition of renormalized solution which details can be found in several papers such as \cite{bebo}, \cite{11DMOP} or \cite{MP2018}.
\subsection{Renormalized solution}
 For each integer $k>0$, and for $s \in \mathbb{R}$ we firstly define the operator $T_k: \mathbb{R} \to \mathbb{R}$ as:
\begin{align}
\label{eq:Tk}
T_k (s) = \max\left\{ -k,\min\{k,s\} \right\},
\end{align}
and this belongs to $W_0^{1,p}(\Omega)$ for every $k>0$, which satisfies
\begin{align*}
-\mbox{div} A(x,\nabla T_k(u)) = \mu_k
\end{align*}
in the sense of distribution in $\Omega$ for a finite measure $\mu_k$ in $\Omega$. 
\begin{definition}
\label{def:truncature}
Let $u$ be a measurable function defined on $\Omega$ which is finite almost everywhere, and satisfies $T_k(u) \in W^{1,1}_0(\Omega)$ for every $k>0$. Then, there exists a unique measurable function $v: \ \Omega \to \mathbb{R}^n$ such that
\begin{align}
\nabla T_k(u) = \chi_{\{|u| \le k\}} v , \quad \text{almost everywhere in} \ \ \Omega, \ \text{for  every} \ k>0.
\end{align}
Moreover, the function $v$ is so-called ``distributional gradient $\nabla u$'' of $u$.
\end{definition}

We define $\mathfrak{M}_b(\Omega)$ as the space of all Radon measures on $\Omega$ with bounded total variation. The positive part, the negative part and total variation of a measure $\mu$ in $\mathfrak{M}_b(\Omega)$ are denoted by $\mu^+, \mu^-$ and $|\mu|$ - is a bounded positive measure on $\Omega$, respectively. For every measure $\mu$ in $\mathfrak{M}_b(\Omega)$ can be written in a unique way as $\mu = \mu_0+\mu_s$, where $\mu_0$ in $\mathfrak{M}_0(\Omega)$ and $\mu_s$ in $\mathfrak{M}_s(\Omega)$. The following Definition \ref{def:renormsol3} of renormalized solution to equation \eqref{eq:elliptictype} was introduced in \cite{11DMOP}, and we reproduce them herein as.

\begin{definition}
\label{def:renormsol3}
Let $\mu = \mu_0+\mu_s \in \mathfrak{M}_b(\Omega)$, where $\mu_0 \in \mathfrak{M}_0(\Omega)$ and $\mu_s \in \mathfrak{M}_s(\Omega)$. A measurable function $u$ defined in $\Omega$ and finite almost everywhere is called a renormalized solution of \eqref{eq:elliptictype} if $T_k(u) \in W^{1,p}_0(\Omega)$ for any $k>0$, $|{\nabla u}|^{p-1}\in L^r(\Omega)$ for any $0<r<\frac{n}{n-1}$, and $u$ has the following additional property. For any $k>0$ there exist  nonnegative Radon measures $\lambda_k^+, \lambda_k^- \in\mathfrak{M}_0(\Omega)$ concentrated on the sets $u=k$ and $u=-k$, respectively, such that $\mu_k^+\rightarrow\mu_s^+$, $\mu_k^-\rightarrow\mu_s^-$ in the narrow topology of measures and  that
 \begin{align*}
 \int_{\{|u|<k\}}\langle A(x,\nabla u),\nabla \varphi\rangle
  	dx=\int_{\{|u|<k\}}{\varphi d}{\mu_{0}}+\int_{\Omega}\varphi d\lambda_{k}%
  	^{+}-\int_{\Omega}\varphi d\lambda_{k}^{-},
 \end{align*}
  	for every $\varphi\in W^{1,p}_0(\Omega)\cap L^{\infty}(\Omega)$.
\end{definition}
 
\subsection{Lorentz spaces}
\label{sec:othersdefs}

For some $s \in (0,\infty)$ and $ t \in (0,\infty]$, the Lorentz space $L^{s,t}(\Omega)$ is defined as the set of all Lebesgue measurable functions $f$ on $\Omega$ such that:
\begin{align}
\label{eq:lorentz}
\|f\|_{L^{s,t}(\Omega)} := \left[ s \int_0^\infty{ \lambda^s \left|  \{x \in \Omega: \, |f(x)|>\lambda\} \right|^{\frac{t}{s}} \frac{d\lambda}{\lambda}} \right]^{\frac{1}{t}} < \infty,
\end{align}
as $t \neq \infty$ and 
\begin{align*}
\|f\|_{L^{s,\infty}(\Omega)} := \sup_{\lambda>0}{\lambda \left|\{x \in \Omega: \, |f(x)|>\lambda\}\right|^{\frac{1}{s}}} < \infty,
\end{align*}
where $|\mathcal{O}|$ denotes the $n$-dimensional Lebesgue measure of a set $\mathcal{O} \subset \mathbb{R}^n$. The space $L^{s,\infty}(\Omega)$ is known as the usual weak $L^s(\Omega)$ or Marcinkiewicz space. 

It is well known that for $t=s$, the Lorentz space $L^{s,s}(\Omega)$ in \eqref{eq:lorentz} is exactly the Lebesgue space $L^s(\Omega)$. Moreover, we have $L^s(\Omega) \subset L^{s,\infty}(\Omega) \subset L^r(\Omega)$, for some $1<r<s<\infty.$

In fact, the quasi-norm $\|\cdot\|_{L^{s,t}(\Omega)}$ may be defined as the other form which is given by Lemma~\ref{lem:f*} below.  For a measure function $f$ in $\Omega$, the distribution function $d_f: \ [0,\infty) \to [0,\infty)$ of $f$ is defined by 
$$ d_f(\lambda) = |\{x\in \Omega: \ |f(x)|>\lambda\}|.$$
The decreasing rearrangement $f^*: \ [0,\infty) \to [0,\infty)$  of $f$ defines as follows
$$ f^*(\lambda) = \inf\{\eta>0: \ d_f(\eta) \le \lambda\}.$$

\begin{lemma}
\label{lem:f*}
Let $s \in (0,\infty)$ and $ t \in (0, \infty]$. For some $f \in L^{s,t}(\Omega)$, there holds
\begin{equation}
\label{eq:normL}
\|f\|_{L^{s,t}(\Omega)} = \begin{cases} \displaystyle \left[\int_0^{\infty} \left(\lambda^{\frac{1}{s}} f^*(\lambda) \right)^t \frac{d\lambda}{\lambda} \right]^{\frac{1}{t}}, & \quad t < \infty, \\ \displaystyle \sup_{\lambda>0} \lambda^{\frac{1}{s}} f^*(\lambda), & \quad t = \infty. \end{cases}
\end{equation}
\end{lemma}
\begin{proof}
The proof of this lemma can be found in~\cite[Proposition 1.4.9]{55Gra}.
\end{proof}
\subsection{A norm in Lorentz space}
We define by $f^{**}: \ [0,\infty) \to [0,\infty)$ the maximal functional of $f$ as follows
\begin{align*}
f^{**}(\lambda) =  \frac{1}{\lambda} \int_0^\lambda f^*(\eta) d\eta,  \ \mbox{ for } \lambda >0 \mbox{ and }   f^{**}(0) = f^*(0).
\end{align*}
For some $ s\in (1, \infty)$, $ t \in [1, \infty]$ and for any $f \in L^{s,t}(\Omega)$, let us introduce 
\begin{align}
\label{eq:Lo_norm}
|||f|||_{L^{s,t}(\Omega)}:=  \left[\int_{0}^{\infty} \left( \lambda^{\frac{1}{s}} f^{**}(\lambda)  \right)^t \frac{d\lambda}{\lambda} \right]^{\frac{1}{t}}, 
\end{align}
if $1\le t<\infty$ and
\begin{align}
\label{eq:Lo_norm_infty}
|||f|||_{L^{s,\infty}(\Omega)}:=  \sup_{\lambda>0} \lambda^{\frac{1}{s}} f^{**}(\lambda).
\end{align}

\begin{lemma}
\label{lem:equi}
Let $ s\in (1, \infty)$ and $ t \in [1, \infty]$. The functional $|||\cdot|||_{L^{s,t}(\Omega)}$ defined by~\eqref{eq:Lo_norm}-\eqref{eq:Lo_norm_infty} is a norm in Lorentz space $L^{s,t}(\Omega)$. Moreover, for any $f \in L^{s,t}(\Omega)$ there holds
\begin{equation}
\label{eq:equi}
 \|f\|_{L^{s,t}(\Omega)} \le |||f|||_{L^{s,t}(\Omega)} \le \frac{s}{s-1}  \|f\|_{L^{s,t}(\Omega)}.
\end{equation}
\end{lemma}
\begin{proof} We prove that the functional $|||\cdot|||_{L^{s,t}(\Omega)}$ defined by~\eqref{eq:Lo_norm}-\eqref{eq:Lo_norm_infty} is a norm in Lorentz space $L^{s,t}(\Omega)$. We remark that
$$ f^{**}(\lambda) = \frac{1}{\lambda} \int_0^\lambda f^*(\eta)d\eta = \frac{1}{\lambda} \sup_{|E|=\lambda}\int_E |f(x)|dx.$$
This deduces the subadditivity of the maximal functional, i.e., for any measurable function $f, g$ and for any $\lambda>0$, there holds
\begin{align*}
(f+g)^{**}(\lambda) & = \frac{1}{\lambda} \sup_{|E|=\lambda}\int_E |f(x)+ g(x)|dx\\
 & \le \frac{1}{\lambda} \sup_{|E|=\lambda}\int_E |f(x)|dx + \frac{1}{\lambda} \sup_{|E|=\lambda}\int_E |g(x)|dx \\
 &= f^{**}(\lambda) + g^{**}(\lambda).
\end{align*}
By the above subadditivity and Minkowski's inequality, it follows that the functional $|||\cdot|||_{L^{s,t}(\Omega)}$ is a norm in Lorentz space $L^{s,t}(\Omega)$.

The first inequality of \eqref{eq:equi} is obtained from Lemma~\ref{lem:f*} and the fact that $f^*(\lambda) \le f^{**}(\lambda)$ for every $\lambda >0$. We then prove the second inequality of \eqref{eq:equi}. 

For any $1<t<\infty$,  by Holder's inequality with $\frac{1}{t} + \frac{1}{t'} = 1$, we obtain
\begin{align}
\label{eq:1}
\nonumber
\left(\int_0^{\lambda} f^*(\eta) d\eta \right)^t & = \left(\int_0^{\lambda} f^*(\eta)\eta^{\frac{1}{s}-\frac{1}{ts}} \eta^{-\frac{1}{s}+\frac{1}{ts}}  d\eta\right)^t \\
\nonumber
& \le \left(\int_0^{\lambda} (f^*(\eta))^t\eta^{\frac{t}{s}-\frac{1}{s}} d\eta\right)  \left(\int_0^{\lambda} \eta^{-\frac{t}{s}+\frac{t'}{ts}} d\eta\right)^{\frac{t}{t'}}\\
\nonumber
  & = \left(\int_0^{\lambda} (f^*(\eta))^t\eta^{\frac{t}{s}-\frac{1}{s}}  d\eta\right)  \left(\int_0^{\lambda} \eta^{-\frac{1}{s}} d\eta\right)^{t-1}\\
 & = \left(\frac{1}{1 - 1/s}\right)^{t-1} \lambda^{(t-1)(1-1/s)} \int_0^{\lambda} \left(f^*(\eta)\right)^t \eta^{\frac{t}{s}-\frac{1}{s}} d\eta,
\end{align}
for any $\lambda>0$. It is easy to see that the inequality~\eqref{eq:1} also holds for $t=1$. By integrating both sides of~\eqref{eq:1} from zero to infinity and using Fubini's Theorem we get that
\begin{align*}
|||f|||_{L^{s,t}(\Omega)} & = \left[\int_{0}^{\infty} \lambda^{\frac{t}{s}-t-1} \left( \int_0^\lambda f^*(\eta) d\eta  \right)^t d\lambda \right]^{\frac{1}{t}} \\
& \le \left[ \left(\frac{1}{1 - 1/s}\right)^{t-1} \int_{0}^{\infty} \lambda^{\frac{1}{s}-2}  \int_0^{\lambda} \left(f^*(\eta)\right)^t \eta^{\frac{t}{s}-\frac{1}{s}} d\eta  d\lambda \right]^{\frac{1}{t}} \\
& = \left[ \left(\frac{s}{s-1}\right)^{t-1} \int_{0}^{\infty} \left(f^*(\eta)\right)^t \eta^{\frac{t}{s}-\frac{1}{s}}   \int_{\eta}^{\infty} \lambda^{\frac{1}{s}-2}   d\lambda d\eta  \right]^{\frac{1}{t}} \\
& = \frac{s}{s-1} \|f\|_{L^{s,t}(\Omega)},
\end{align*}
which deduces the second inequality for $t \in [1, \infty)$. In the case of $t = \infty$, we also have

\begin{align*}
|||f|||_{L^{s,\infty}(\Omega)} & = \sup_{\lambda>0} \lambda^{\frac{1}{s}-1} \int_0^{\lambda} \eta^{-\frac{1}{s}} \eta^{\frac{1}{s}} f^{*}(\eta)d\eta \\
& \le \sup_{\lambda>0} \lambda^{\frac{1}{s}-1} \left(\int_0^{\lambda} \eta^{-\frac{1}{s}} d\eta\right)  \|f\|_{L^{s,\infty}(\Omega)} \\
& = \frac{s}{s-1} \|f\|_{L^{s,\infty}(\Omega)}.
\end{align*}
\end{proof}
\subsection{Lorentz-Morrey spaces}
Let $ s \in (0,\infty)$, $t \in (0,\infty]$ and $\kappa \in (0,n]$. The  Lorentz-Morrey functional spaces $L^{s,t;\,\kappa}(\Omega)$ is the set of all functions $g \in L^{s,t}(\Omega)$ such that
\begin{align}
\label{eq:LMsp}
\|f\|_{L^{s,t;\,\kappa}(\Omega)}:=\sup_{0<\rho\le diam(\Omega); \, x \in \Omega}{\rho^{\frac{\kappa-n}{s}}}\|f\|_{L^{s,t}(B_\rho(x)\cap\Omega)} < \infty,
\end{align}
where $B_\rho(x)$ denotes the ball centered $x$ with radius $\rho$ in $\mathbb{R}^n$.

Obviously, in the case of $\kappa = n$ the Lorentz-Morrey space $L^{s,t;\,\kappa}(\Omega)$ is exactly the Lorentz space $L^{s,t}(\Omega)$. It is similar to the Lorentz space, the functional $\|\cdot\|_{L^{s,t;\,\kappa}(\Omega)}$ is just a quasi-norm in general. So we need to define a norm in Lorentz-Morrey space. With this norm, the set $V_{\varepsilon}$ defined by~\eqref{eq:setV} in the next section will be convex.

Let $ s \in (1,  \infty)$, $ t \in [1, \infty]$ and $\kappa \in (0,n]$. For any $f \in L^{s,t;\,\kappa}(\Omega)$, let us set
\begin{align}
\label{eq:normLM}
|||f|||_{L^{s,t;\,\kappa}(\Omega)}:=\sup_{0<\rho\le diam(\Omega); \, x \in \Omega}{\rho^{\frac{\kappa-n}{s}}}|||f|||_{L^{s,t}(B_\rho(x)\cap\Omega)}.
\end{align}

The following corollary is directly obtained by definition~\eqref{eq:normLM} and Lemma~\ref{lem:equi}.
\begin{corollary}
\label{coro:equi}
Let $ s \in (1,  \infty)$, $ t \in [1, \infty]$ and $\kappa \in (0,n]$. The function $|||\cdot|||_{L^{s,t;\,\kappa}(\Omega)}$ defined by~\eqref{eq:normLM} is a norm in Lorentz-Morrey space $L^{s,t;\,\kappa}(\Omega)$. Moreover, for any $f \in L^{s,t;\,\kappa}(\Omega)$, there holds
\begin{equation}
 \|f\|_{L^{s,t;\,\kappa}(\Omega)} \le |||f|||_{L^{s,t;\,\kappa}(\Omega)} \le \frac{s}{s-1}  \|f\|_{L^{s,t;\,\kappa}(\Omega)}.
\end{equation}
\end{corollary}

\section{Proof of main theorem}
 In this section, we give the detail proof of Theorem~\ref{theo:main}. The main idea of our proof is based on applying Schauder Fixed Point Theorem (see \cite{Trudinger}) for a continuous mapping $T: \ V_{\varepsilon} \to V_{\varepsilon}$, where $V_{\varepsilon}$ is closed, convex and $T(V_{\varepsilon})$ is precompact under the strong topology of $W_0^{1,1}(\Omega)$. The proof is divided into four steps under all hypotheses of Theorem~\ref{theo:main}.\\

\noindent 
\begin{proof}[{\bf Proof of Theorem \ref{theo:main}}] Let $q,\, s,\, t$ satisfying~\eqref{eq:cond_q}, \eqref{eq:cond_s} and set $\theta = \frac{q}{q-p+1}$ as in Theorem \ref{theo:main}. 
For every $\varepsilon>0$, we consider the set $V_{\varepsilon}$ as follows
\begin{align}
\label{eq:setV}
V_{\varepsilon}=\left\{u\in W^{1,1}_0(\Omega): \ |||\nabla u|||_{L^{qs,qt;\,\theta s}(\Omega)}\leq \varepsilon \right\}.
\end{align}
We introduce the mapping $T$ as 
\begin{equation}
\label{eq:T}
T: \ V_{\varepsilon} \to V_{\varepsilon} \ \mbox{ defined by } \ T(v) = u, \ \mbox{ for any } \ v\in V_{\varepsilon},
\end{equation}
where $u$ is the unique renormalized solution to the following equation 
\begin{equation}
\label{eq:elliptictype11eeaaae1}
\begin{cases}
-\mbox{div}(A(x,\nabla u)) &=\ |\nabla v|^q+ \mu \quad \text{in} \ \ \Omega, \\
\hspace{2cm} u &= \ 0 \quad \text{on} \ \ \partial \Omega.
\end{cases}
\end{equation}
{\bf{First step:}} {\it $V_{\varepsilon}$ is closed and convex under the strong topology of $W_0^{1,1}(\Omega)$.} 
 
We first prove that $V_{\varepsilon}$ is convex. Indeed, for any $u, v \in V_{\varepsilon}$ and $\eta \in [0,1]$, we must to show that $w = \eta u+(1-\eta)v\in V_{\varepsilon}$. We remark that $|||\cdot|||_{L^{s,t}(\mathcal{O})}$ is a norm in Lorentz-Morrey space $L^{s,t}(\mathcal{O})$, for any subset $\mathcal{O}$ of $\Omega$. Therefore, for any $z \in \Omega$ and $ 0 < \rho \le \mbox{diam}(\Omega)$, we have
\begin{align*}
|||\nabla w|||_{L^{s,t}(B_{\rho}(z)\cap \Omega)}  \ \le \ \eta  |||\nabla u|||_{L^{s,t}(B_{\rho}(z)\cap \Omega)} + (1-\eta)  |||\nabla v|||_{L^{s,t}(B_{\rho}(z)\cap \Omega)}.
\end{align*}
Multiply two sides of this inequality by $\rho^{\frac{\kappa-n}{s}}$, we obtain 
\begin{align*}
 \rho^{\frac{\kappa-n}{s}}|||\nabla w|||_{L^{s,t}(B_{\rho}(z)\cap \Omega)}  \ \le \ \eta \rho^{\frac{\kappa-n}{s}} |||\nabla u|||&_{L^{s,t}(B_{\rho}(z)\cap \Omega)} \\ 
 &+ (1-\eta) \rho^{\frac{\kappa-n}{s}} |||\nabla v|||_{L^{s,t}(B_{\rho}(z)\cap \Omega)},
\end{align*}
which deduces that
$$
 |||\nabla w|||_{L^{s,t;\,\kappa}(\Omega)}  \ \le \ \eta  |||\nabla u|||_{L^{s,t;\,\kappa}(\Omega)} + (1-\eta)  |||\nabla v|||_{L^{s,t;\,\kappa}(\Omega)}   \le \varepsilon,
$$
which gives $w \in V_{\varepsilon}$. 

Next we show that $V_{\varepsilon}$ is closed under the strong topology of $W^{1,1}_0(\Omega)$. Let $\{u_k\}_{k \in \mathbb{N}}$ be a sequence in $V_{\varepsilon}$ such that $u_k$ converges strongly in $W^{1,1}_0(\Omega)$ to a function $u$. Let $z \in \Omega$ and $ 0 < \rho \le \mbox{diam}(\Omega)$, we note that $\nabla u_k$ converges to $\nabla u$ almost everywhere in $ B_{\rho}(z)\cap \Omega$.  By~\cite[Proposition 1.4.9]{55Gra}, it follows that the sequence $(\nabla u_k)^*$ converges to $(\nabla u)^*$ in $[0,\infty)$. For any $\lambda>0$, by the Fatou lemma,  we obtain that
\begin{equation*}
\displaystyle \frac{1}{\lambda} \int_0^{\lambda} (\nabla u)^*(\eta)d\eta \le \limsup_{k\to \infty} \frac{1}{\lambda} \int_0^{\lambda} (\nabla u_k)^*(\eta)d\eta,
\end{equation*}
which asserts that
\begin{equation*}
(\nabla u)^{**}(\lambda) \le \limsup_{k\to \infty} (\nabla u_k)^{**}(\lambda).
\end{equation*}
 We thus get
\begin{align*}
{\rho^{\frac{\kappa-n}{s}}}|||\nabla u|||_{L^{s,t}(B_\rho(z)\cap\Omega)} & \le  \limsup_{k\to \infty} {\rho^{\frac{\kappa-n}{s}}} |||\nabla u_k|||_{L^{s,t}(B_\rho(z)\cap\Omega)} \\
& \le |||\nabla u_k|||_{L^{s,t;,\kappa}(\Omega)} \le \varepsilon.
\end{align*}
It follows that 
$$ |||\nabla u|||_{L^{s,t;\kappa}(\Omega)}=\sup_{0<\rho\le \mbox{diam}(\Omega), \, z \in \Omega} {\rho^{\frac{\kappa-n}{s}}}|||\nabla u|||_{L^{s,t}(B_\rho(z)\cap\Omega)} \le \varepsilon,$$
which leads to $u \in V_{\varepsilon}$.

\noindent
{\bf{Second step:}} {\it There exist $\delta_0>0$ and $\varepsilon_0>0$ such that if
$\|\mu\|_{L^{s,t;\,\theta s}(\Omega)}\leq \delta_0$ then the mapping $T: V_{\varepsilon_0}\to V_{\varepsilon_0}$ in~\eqref{eq:T} is well-defined.}

Under the hypotheses~\eqref{eq:cond_q} and~\eqref{eq:cond_s}, by Corollary \ref{coro:P}, there exists a positive constant $C$ such that for any renormalized solution $u$ to equation~\eqref{eq:elliptictype}, there holds
\begin{align}
\label{eq:C}
\|\nabla u\|^{p-1}_{L^{qs,qt;\,\theta s}(\Omega)} \le C\|\mu\|_{L^{s,t;\,\theta s}(\Omega)}.
\end{align}
We first prove that there exists $\delta_0>0$ such that if $ \|\mu\|_{L^{s,t;\,\theta s}(\Omega)}\leq \delta_0 $
then there exists a positive number $y_0$ satisfying
\begin{equation}
\label{eq:y0}
\frac{Cs}{s-1} \left(\frac{qs}{qs-1}\right)^{p-1}\left(y_0+\|\mu\|_{L^{s,t;\,\kappa}(\Omega)}\right) = y_0^{\frac{p-1}{q}}. 
\end{equation}
 We consider the function $g: \ [0,\infty) \to \mathbb{R}$ defined by
\begin{equation}
\label{eq:f}
g(y) = (cy+ca)^{\frac{\theta}{\theta-1}} - y,
\end{equation}
with  $c = \displaystyle \frac{Cs}{s-1} \left(\frac{qs}{qs-1}\right)^{p-1}$ and $a = \|\mu\|_{L^{s,t;\,\theta s}(\Omega)}$. Noting that $\theta>1$, let us choose 
$$\delta_0 = \frac{1}{c\theta}\left(\frac{\theta-1}{c\theta}\right)^{\theta-1} >0.$$ 
If $a \le \delta_0$ then the function $g$ given by~\eqref{eq:f} satisfies $g(0)>0$ and $\displaystyle \lim_{y\to \infty} g(y) = \infty$. Moreover, $g'(y) = \frac{\theta c}{\theta-1}(cy+ca)^{\frac{1}{\theta-1}}-1$, thus $g'(y) = 0$ if and only if $y = y^*$ given by
$$ y^* = \frac{1}{c} \left(\frac{\theta-1}{c\theta }\right)^{\theta-1} - a =  {\theta\delta_0-a}>0. $$
It follows that the minimum value of $g$ on $[0,\infty)$ is 
$$g(y^*) = (cy^* + ca)\frac{\theta-1}{c\theta } - y^* = {a - \delta_0}  \le 0.$$
For this reason, we conclude that $g$ has exactly one root $y_0 \in (0,y^*]$ which satisfies~\eqref{eq:y0}.

Let us set $ \varepsilon_0 = y_0^{\frac{1}{q}}$. By the definition of $T$, for any $v \in V_{\varepsilon_0}$, $u = T(v) \in W_0^{1,1}(\Omega)$ is the unique renormalized solution to equation~\eqref{eq:elliptictype11eeaaae1} (see \cite{11DMOP} for the uniqueness of renormalized solution to~\eqref{eq:elliptictype11eeaaae1}). Applying~\eqref{eq:C} and Corollary~\ref{coro:equi}, we obtain
\begin{align}
\label{eq:E3}
\|\nabla u\|^{p-1}_{L^{qs,qt;\,\theta s}(\Omega)} \le C  \| |\nabla v|^q+\mu\|_{L^{s,t;\,\theta s}(\Omega)} \le C  ||| |\nabla v|^q+\mu|||_{L^{s,t;\,\theta s}(\Omega)}.
\end{align}
Combining~\eqref{eq:E3} with the triangle inequality and Corollary~\ref{coro:equi}, one has
\begin{align}
\nonumber
|||\nabla u|||^{p-1}_{L^{qs,qt;\,\theta s}(\Omega)} & \le \left(\frac{qs}{qs-1}\right)^{p-1}\|\nabla u\|^{p-1}_{L^{qs,qt;\,\theta s}(\Omega)} \\
\nonumber
& \le C\left(\frac{qs}{qs-1}\right)^{p-1}\left[|||(|\nabla v|^q)|||_{L^{s,t;\,\theta s}(\Omega)} + |||\mu|||_{L^{s,t;\,\theta s}(\Omega)}  \right] \\ 
\nonumber
& \le  \frac{Cs}{s-1} \left(\frac{qs}{qs-1}\right)^{p-1} \left[\|\nabla v\|^q_{L^{qs,qt;\,\theta s}(\Omega)} + \|\mu\|_{L^{s,t;\,\theta s}(\Omega)}  \right] \\ 
\label{eq:E4}
& \le \frac{Cs}{s-1} \left(\frac{qs}{qs-1}\right)^{p-1} \left[|||\nabla v|||^q_{L^{qs,qt;\,\theta s}(\Omega)} + \|\mu\|_{L^{s,t;\,\theta s}(\Omega)}  \right].
\end{align}
Here we note that $|||\nabla v|||^q_{L^{qs,qt;\,\theta s}(\Omega)} \le y_0$, with $y_0$ is the root of~\eqref{eq:y0} and $\varepsilon_0 = y_0^{\frac{1}{q}}$. Therefore, we can rewrite~\eqref{eq:E4} as
\begin{align*}
|||\nabla u|||^{p-1}_{L^{qs,qt;\,\theta s}(\Omega)} \le  y_0^{\frac{p-1}{q}}  = \varepsilon_0^{p-1},
\end{align*}
which yields $T(v) = u \in V_{\varepsilon_0}$. We conclude that the mapping $T$ is well-defined.

\noindent
{\bf{Third step:}} {\it $T:V_{\varepsilon_0}\to V_{\varepsilon_0}$ is continuous, and $\overline{T(V_{\varepsilon_0})}$ is a compact set under the strong topology of $W^{1,1}_0(\Omega)$.}

Let $\{v_k\}_{k \in \mathbb{N}}$ be a sequence in $V_{\varepsilon_0}$ such that $v_k$ converges strongly in $W^{1,1}_0(\Omega)$ to a function $v \in V_{\varepsilon_0}$. For every $k \in \mathbb{N}$, $u_k = T(v_k)$ is the renormalized solution of the equation
\begin{align}
\label{eq:uk}
\begin{cases} -\mbox{div}(A(x,\nabla u_k)) & = \ |\nabla v_k|^q + \mu \quad \mbox{in} \ \Omega, \\ \hspace{2cm} u_k & = \ 0 \quad \mbox{on} \ \partial \Omega,\end{cases} 
\end{align}
with
\begin{equation}
\label{eq:nablavk}
\|\nabla v_k\|_{L^{qs,qt;\,\theta s}(\Omega)} \le \varepsilon_0.
\end{equation}
We obtain that 
\begin{equation}
\label{eq:nablavk2}
\|\nabla v_k\|_{L^{r}(\Omega)} \le \varepsilon_0,
\end{equation}
for any $q<r<qs$. Hence, there exists a subsequence $\{v_{k_j}\}_{j\in \mathbb{N}}$ of $\{v_k\}$ such that $\nabla v_{k_j}$ converges to $\nabla v$ almost everywhere in $\Omega$. By \eqref{eq:nablavk2} and Vitali Convergence Theorem we have $\nabla v_{k_j}$ converges to $\nabla v$ strongly in $L^q(\Omega)$. This follows that  $\nabla v_{k}$ converges to $\nabla v$ strongly in $L^q(\Omega)$. 

By the stability result of renormalized solution in~\cite[Theorem 3.4]{11DMOP}, there exists a subsequence $\{u_{k_j}\}$ such that $\{u_{k_j}\}$ converges to $u$ almost everywhere in $\Omega$, where $u$ is the unique renormalized solution  of the following equation
\begin{align*}
\begin{cases} -\mbox{div}(A(x,\nabla u)) & = \ |\nabla v|^q + \mu \quad \mbox{in} \ \Omega, \\ \hspace{2cm} u & = \ 0 \quad \mbox{on} \ \partial \Omega.\end{cases} 
\end{align*}
Moreover, $\nabla u_{k_j}$ also converges to $\nabla u$ almost everywhere in $\Omega$. It is similar to the above, using again Vitali Convergence Theorem with the facts that $qs>1$ and
$$\|\nabla u_{k_j}\|_{L^{qs,qt;\,\theta s}(\Omega)} \le \varepsilon_0,$$ we deduce that $u_k$ converges strongly to $u$ in $W_0^{1,1}(\Omega)$. It follows that $T$ is continuous.

The compactness of the set $\overline{T(V_{\varepsilon_0})}$ under the strong topology of $W^{1,1}_0(\Omega)$ can be proved by the same method as in the above. Indeed, let $\{u_k\} = \{T(v_k)\}$ be a sequence in $T(V_{\varepsilon_0})$ where $\{v_k\} \subset V_{\varepsilon_0}$, then we have \eqref{eq:uk}, \eqref{eq:nablavk}. Applying~\cite[Theorem 3.4]{11DMOP} again, there exist a subsequence $\{u_{k_j}\}$ and a function $u \in W_0^{1,1}(\Omega)$ such that $\nabla{u_{k_j}} \to \nabla u$ almost everywhere in $\Omega$. Finally, using Vitali Convergence Theorem again, we obtain that $\{u_{k_j}\}$ strongly converges to $u$ in $ W_0^{1,1}(\Omega)$.

\noindent
{\bf{Fourth step:}} {\it Applying Schauder Fixed Point Theorem.}

By Schauder Fixed Point Theorem, the mapping $T: \ V_{\varepsilon_0} \to V_{\varepsilon_0}$ has a fixed point $u$ in $V_{\varepsilon_0}$. This gives a solution $u$ to equation \eqref{eq:Riccati}. Moreover, in the proof of the second step, we obtain the following estimation
\begin{align*}
\|\nabla u\|^q_{L^{qs,qt;\,\theta s}(\Omega)} \le |||\nabla u|||^q_{L^{qs,qt;\,\theta s}(\Omega)} \le y^* \le  \theta\delta_0 - \|\mu\|_{L^{s,t;\,\theta s}(\Omega)}. 
\end{align*} 
The proof of Theorem~\ref{theo:main} is complete.
\end{proof}

\section*{Acknowledgments}
The second author was supported by Ho Chi Minh City University of Education.

\end{document}